\documentclass{amsart}[12pt]

\usepackage{amscd,amssymb}
\usepackage{comment}
\usepackage{pgf, tikz}

\usepackage{placeins}

\usepackage{color}

\newtheorem{theorem}{Theorem}[section]

\newtheorem{proposition}[theorem]{Proposition}
\newtheorem{corollary}[theorem]{Corollary}

\theoremstyle{definition}

\newcommand{\GL} {\mathrm{GL}}
\newcommand{\SL} {\mathrm{SL}}

\newcommand{\var} {\operatorname{var}}

\def\ZZ {{\mathbb Z}}     

\begin{document}
	
	\title[On the cocharacter sequence of some PI-algebras]
	{On the cocharacter sequence of some PI-algebras}
	
	\author[Elitza Hristova]
	{Elitza Hristova}
	
	\address{Institute of Mathematics and Informatics,
	Bulgarian Academy of Sciences,
	Acad. G. Bonchev Str., Block 8,
	1113 Sofia, Bulgaria}
	\email{e.hristova@math.bas.bg}
	
	\thanks{Partially supported by Grant KP-06 N92/2 of the Bulgarian National Science Fund.}
	
	
	
	\begin{abstract}
		
		Let $A$ be a unital associative PI-algebra over a field of characteristic zero. We study which partitions $\lambda$ appear with nonzero multiplicities in the cocharacter sequence of $A$ for several classes of algebras $A$. Berele defines the eventual arm width $\omega_0(A)$ to be the maximal integer $d$ so that if $\lambda$ appears with nonzero multiplicity in the cocharacter sequence of $A$, then $\lambda$ can have at most $d$ parts arbitrarily large. Berele also shows that if $A$ is Lie nilpotent, then $\omega_0(A) = 1$. In the first part of this paper, we show that if $A$ is unital, then $\omega_0(A) = 1$ if and only if $A$ is Lie nilpotent. To prove this statement, we show that the algebra of proper polynomials $B_n(A)$ is finite dimensional if and only if $A$ is Lie nilpotent. In the second part, we give a bound on the nonzero multiplicities $\lambda$ in the cocharacter sequence of $A$, when the T-ideal of identities of $A$ is equal to a product of T-ideals generated by long commutators. As an application, we show that for a Lie nilpotent algebra $A$, the nonzero multiplicities $m_{\lambda}(A)$ correspond to partitions $\lambda$ which are supported in step-like diagrams in which the number of steps grows with the index of Lie nilpotency. Finally, we give also some applications to the noncommutative invariant theory of the special linear group $\SL(n)$.
\end{abstract}	
	\maketitle
	\section{Introduction}
	
	In this paper, all algebras are associative and unital and $K$ is a field of characteristic zero. Let $K \left \langle X \right\rangle$ denote the free associative algebra generated by a countably infinite set $X = \{x_1, x_2, \dots\}$ and let $K\left\langle X_m \right\rangle$ denote the free associative algebra generated by the finite subset $X_m = \{x_1, \dots, x_m\} \subset X$. A T-ideal in $K \left \langle X \right\rangle$ is any two-sided associative ideal that is closed under all $K$-algebra endomorphisms of $K \left \langle X \right\rangle$. If $A$ is a PI-algebra, then the set of all polynomial identities of $A$ forms a T-ideal, which we denote by $T(A)$. Let $P_m$ denote the subspace of multilinear polynomials of degree $m$ in $K\left\langle X_m \right\rangle$. 
	Then $P_m$ is naturally a module over the symmetric group $S_m$, where the action of $S_m$ is given by permuting the variables. For any PI-algebra $A$, one also considers the natural quotient map
	\[
	K\left\langle X_m \right\rangle \rightarrow K\left\langle X_m \right\rangle/ (K\left\langle X_m \right\rangle \cap T(A))	
	\]
	and sets $P_m(A)$ to be the image of $P_m$ under the above map.	
	Then $P_m(A)$ inherits the $S_m$-module structure of $P_m$. The character of $P_m(A)$ as an $S_m$-module is denoted by $\chi_m(A)$ and is called the $m$-th cocharacter of the algebra $A$ (or of the T-ideal $T(A)$). One important question in the quantitative theory of PI-algebras is to determine the sequence of cocharacters of a given PI-algebra $A$, i.e. to determine for every $m$ the multiplicities $m_{\lambda}(A)$ in the decomposition
	
	\begin{align} \label{eq_cochar}
		\chi_m(A) = \sum_{\lambda \vdash m} m_{\lambda}(A) \chi_{\lambda},
	\end{align}
	where $\lambda$ is a partition of $m$ and $\chi_{\lambda}$ is the character of the irreducible $S_m$-module corresponding to the partition $\lambda$.
	
	The multiplicities $m_{\lambda}(A)$ from (\ref{eq_cochar}) are explicitly known only for a few classes of PI-algebras $A$. 
	Furthermore, there are general results about the properties of the partitions $\lambda$ such that $m_{\lambda}(A) \neq 0$. In particular, the following classical theorem of Amitsur and Regev (known also as the Hook Theorem) shows that for an arbitrary PI-algebra $A$ we can capture all partitions $\lambda$ for which the multiplicity $m_{\lambda}(A)$ is non-zero in hook Young diagrams. We first define for arbitrary integers $k,l$ the set 
	\[
	H(k,l; m) = \{ \lambda \vdash m : \lambda_{k+1} \leq l\},
	\]
	
	and 
	\[
	H(k,l) = \bigcup_{m \geq 0} H(k,l;m).
	\]
Then, the following theorem holds.
	\begin{theorem}\label{thm_AmitzurRegev} \cite{AR}
		For any PI-algebra $A$ there exists $k$ and $l$ such that all nonzero multiplicities in the cocharacter sequence of $A$ lie in the $k$ by $l$ hook, i.e., if $m_{\lambda}(A) \neq 0$, then $\lambda \in H(k,l)$. 
	\end{theorem}
	
	Pictorially, Theorem \ref{thm_AmitzurRegev} implies that all nonzero multiplicities $m_{\lambda}(A)$ correspond to partitions $\lambda$ that are supported in hook diagrams as in Figure \ref{fig_AmitzurRegev}. One usually calls the first $k$ rows of the Young diagram of the partition $\lambda$ the arm of the partition and the remaining part -- the leg of the partition.
	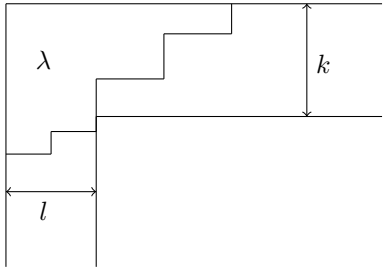
\begin{figure} [!htb] \label{fig_AmitzurRegev}
		\centering
		\begin{tikzpicture} 
			\draw (0,0.5)--(5,0.5);
			\draw (1.2, -1) -- (5, -1);

			\draw (0, 0.5)--(0,-3);
			\draw (1.2,-1)--(1.2,-3);
			\draw [<->] (0,-2)--(1.2,-2);
			\draw (0.5,-2.5) node[above] {$l$};
			
			\draw [<->] (4,0.5)--(4,-1);
			\draw (4,-0.3) node[right] {{$k$}};
			
			\draw (3, 0.5) -- (3, 0.1);
			\draw (2.1, 0.1) -- (3, 0.1);
			\draw (2.1, 0.1) -- (2.1, -0.5);
			\draw (1.2, -0.5) -- (2.1, -0.5);
			\draw (1.2, -0.5) -- (1.2, -1.2);
			\draw (0.6, -1.2) -- (1.2, -1.2);
			\draw (0.6, -1.2) -- (0.6, -1.5);
			\draw (0, -1.5) -- (0.6, -1.5);
			
			\draw (0.5, 0) node[below] {{$\lambda$}};

		\end{tikzpicture}
		
		\caption{Hook diagrams}
	\end{figure}
	\FloatBarrier
	
	In \cite{Be}, Berele defines for any PI-algebra $A$ the eventual arm width $\omega_0(A)$ in the following way: $\omega_0(A)$ = the maximum integer $d$ such that all partitions $\lambda$ which appear with non-zero multiplicity in the cocharacter sequence of $A$ can have at most $d$ parts arbitrarily large. The definition of the eventual leg width $\omega_1(A)$ is similar, using the conjugate partition $\lambda'$. Namely, $\omega_1(A)$ = the maximum integer $h$ such that there exists $\lambda$ in the cocharacter sequence of $A$ with $m_{\lambda}(A) \neq 0$ and such that $\lambda'$ can have $h$ arbitrarily large parts. In other words, $\omega_0(A)$ and $\omega_1(A)$ are minimal such that all nonzero multiplicities correspond to partitions that are supported in diagrams as in Figure 2 
	for some $s_1(A)$ and $s_2(A)$. Moreover, Berele shows that for any $A$, $\omega_0(A) \geq \omega_1(A)$.
	
	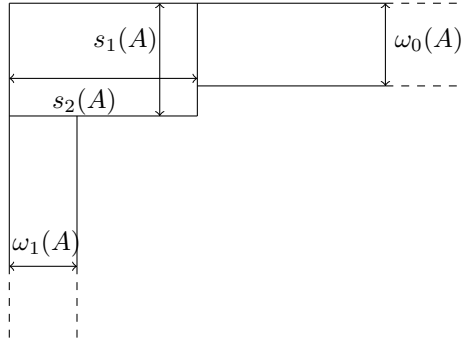
\begin{figure}[!htb] \label{fig_Berele}
		\centering
		\begin{tikzpicture}
			\draw (0,0.5)--(5,0.5);
			\draw [dashed] (5,0.5) -- (6,0.5);
			\draw (2.5, -0.6) -- (5, -0.6);
			\draw [dashed] (5, -0.6) -- (6, -0.6);
			
			\draw [<->] (5,0.5)--(5,-0.6);
			\draw (5,0) node[right] {{$\omega_0(A)$}};

			\draw (2.5, 0.5) -- (2.5, -1);
			\draw (0, -1) -- (2.5, -1);
			
			\draw[<->] (2, 0.5) -- (2, -1);
			\draw (2.1,0) node[left] {{$s_1(A)$}};
			
			\draw[<->] (0, -0.5) -- (2.5, -0.5);
			\draw (1,-0.5) node[below] {{$s_2(A)$}};
			
			\draw (0, 0.5)--(0,-3);
			\draw (0.9,-1)--(0.9,-3);
			\draw [dashed] (0,-3) -- (0, -4);
			\draw [dashed] (0.9, -3) -- (0.9, -4);
			
			\draw [<->] (0,-3)--(0.9,-3);
			\draw (0.5,-3) node[above] {{$\omega_1(A)$}};

		\end{tikzpicture}
		
		\caption{Definitions of $\omega_0(A)$, $\omega_1(A)$, $s_1(A)$, and $s_2(A)$}
	\end{figure}
	\FloatBarrier
	
	Figure 2 shows that $s_2(A)$ is the minimum integer $r_2$ such that for all partitions $\lambda$ with $m_{\lambda}(A) \neq 0$ it holds that $\lambda_{\omega_0(A) +1} \leq r_2$. The integer $s_1(A)$ is defined similarly using the conjugate partition $\lambda'$ and $\omega_1(A)$.
	
	In the same paper \cite{Be}, Berele computes the values of $\omega_i(A)$ ($i=0,1$) for various algebras $A$. Below, we give some cases which are of relevance to our paper. 
	First we define inductively left-normed long commutators in $K\left\langle X \right\rangle$ in the usual way: $[u_1, u_2] = u_1u_2- u_2u_1$
	and for $i \geq 3$, we have $[u_1, \dots, u_i] = [[u_1, \dots, u_{i-1}], u_i]$, where $u_1, \dots, u_i$ are arbitrary elements from $K\left\langle X \right\rangle$. The commutator $[u_1, \dots, u_i]$ is called a commutator of length $i$ or simply an $i$-commutator.
	Let $I_p$ denote the two-sided associative ideal in $K\left\langle X \right\rangle$ generated by all commutators of length $p$. Equivalently, $I_p$ is the $T$-ideal generated by the polynomial $[x_1, \dots, x_p] = 0$. Then, the following theorem holds.
	
	\begin{theorem} \cite{Be} \label{thm_Berele}
		Let $I$ be a T-ideal and let 
		$T(A) = I$.
		Let $k$ be a positive integer written as $k = 3q + r$ where $r= 0,1,2$. If the ideal $I$ is generated by the $k$-th power of a double commutator, then $\omega_0(A) = k$ and $\omega_1(A) = 2q + \lfloor r/2 \rfloor$. If $I$ is generated by the $k$-th power of a commutator of length $p+1$ for $p \geq 2$, then $\omega_0(A) = \omega_1(A) = k$.
	\end{theorem}
	
	In particular, Theorem \ref{thm_Berele} implies that if $A$ is a PI-algebra with $T(A) = I_{p+1}$ for some integer $p \geq 2$, then $\omega_0(A) = \omega_1(A) = 1$.
	
	Furthermore, in the paper \cite{H}, we show that if $T(A) = I_{p+1}$ then $s_2(A) = p-1$. Hence, if $m_{\lambda}(A) \neq 0$ then $\lambda$ belongs to the hook $H(1, p-1)$.
	
	It was independently shown by Berele and Drensky that the question of finding the multiplicities $m_{\lambda}(A)$ in the cocharacter sequence (\ref{eq_cochar}) has an equivalent formulation in the language of the representation theory of the group $\GL(n)$. Namely, we have the following. Let $V$ be an $n$-dimensional $K$-vector space with basis given by the set $X_n = \{x_1, \dots, x_n\}$. Then the tensor algebra $T(V)$ is identified in a natural way with the free associative algebra $K\left \langle X_n \right\rangle$. The group $\GL(n):=\mathrm{GL}(V,K)$ acts on the vector space $V$ with basis $X_n$ and this action is extended to the diagonal action of $\GL(n)$ on $K\left \langle X_n \right\rangle$. Let $I$ be a T-ideal in $K\left \langle X \right\rangle$ and let $\mathfrak{V}_I = \var(I)$ be the variety of algebras satisfying the identities from $I$. The quotient algebra $F_n(\mathfrak{V}_I) = K \left\langle X_n \right\rangle / (K \left\langle X_n \right\rangle \cap I)$ is called the relatively free algebra of rank $n$ in the variety $\mathfrak{V}_I$. 
	$F_n(\mathfrak{V}_I)$ inherits the $\GL(n)$-module structure of $K \left\langle X_n \right\rangle$. In this setting, we are interested in the decomposition of $F_n(\mathfrak{V}_I)$ as a $\GL(n)$-module for different classes of $T$-ideals $I$. In other words, if $\lambda = (\lambda_1 \geq \lambda_2 \geq \dots \geq \lambda_n \geq 0)$ is a non-negative integer partition with at most $n$ parts and if $V_{\lambda}$ denotes the irreducible $\GL(n)$-module with highest weight $\lambda$ we would like to find the multiplicities $m_{\lambda}(I)$ in the following expression:
	\begin{align}\label{eq_decomp_GL}
		F_n(\mathfrak{V}_I) \cong_{\GL(n)} \bigoplus_{\lambda = (\lambda_1, \dots, \lambda_n)} m_{\lambda}(I) V_{\lambda}.
	\end{align}
	
	It follows from the results of Berele and Drensky (see, e.g., \cite{Be2}, \cite{D}) that if $\lambda$ is a partition of $m$ with at most $n$ parts and if $I = T(A)$, then the multiplicity $m_{\lambda}(I)$ from Equation (\ref{eq_decomp_GL}) equals the multiplicity $m_{\lambda}(A)$ from Equation (\ref{eq_cochar}).
	
	In this paper, using the language of $\GL(n)$-modules, we compute the integers $\omega_0(A)$, $\omega_1(A)$ and $s_2(A)$ for several classes of PI-algebras $A$. More precisely, in Corollary \ref{coro_w0} we show that for any associative unital PI-algebra $A$, we have that $\omega_0(A) = 1$ if and only if $T(A) \supseteq I_{p}$ for some positive integer $p$.
	Furthermore, in Corollary \ref{coro_multProduct} we show that if $I = I_{p_1+1} \cdots I_{p_k+1}$ and $T(A) = I$, then $\omega_0(A) = \omega_1(A) =  k$ and $s_2(A) = p_1 + \dots + p_k -1$. In addition, in Proposition \ref{prop_bound_lambda} we show that if again $I = I_{p_1+1} \cdots I_{p_k+1}$ and  $m_{\lambda}(I) \neq 0$, then $\lambda_{2k} \leq p_1 + \dots + p_k - k$. 
	
Using these results, in Section \ref{sec_app}, we obtain a refined condition on the non-zero multiplicities $m_{\lambda}(I_{p+1})$ for any $p \geq 1$. This refined condition is stated in Proposition \ref{prop_bound_Ip}. It says that if $m_{\lambda}(I_{p+1}) \neq 0$, then, for each $1 \leq k \leq \lfloor \frac{n}{2} \rfloor $ such that $p > 2(k-1)$ it holds that $\lambda_{2k} \leq p-k$. In particular, this means that the partitions $\lambda$ with $m_{\lambda}(I_{p+1}) \neq 0$ are contained in a step-like Young diagram and the number of steps in this Young diagram grows with $p$. When $p = 1,2$ we have only the condition $\lambda_2 \leq p-1$, i.e., the Young diagram which captures all non-zero multiplicities has the form of a hook. If $p=3,4$, we have two conditions $\lambda_2 \leq p-1$ and $\lambda_4 \leq p-2$, so we already have a step-like Young diagram. For $p = 5,6$ we have three conditions on $\lambda$, namely $\lambda_2 \leq p-1$, $\lambda_4 \leq p-2$, and $\lambda_6 \leq p-3$ and so on. We also give some restrictions on the conjugate partition $\lambda'$.
	
In Section \ref{sec_invariantTheory}, we mention another application related to the algebra of $\SL(n): = \SL(V, K)$-invariants in $F_n(\mathfrak{V}_I)$, for any $T$-ideal $I$. We give a criterium for finite-dimensionality of the algebra $(F_n(\mathfrak{V}_I))^{\SL(n)}$ in terms of the numbers $\omega_0$, $\omega_1$, $s_1$ and $s_2$. 
	
	\section{On the eventual arm width of PI-algebras}
	In this section, using the language of representation theory of $\GL(n)$, we discuss properties of the eventual arm width of unital associative PI-algebras. We start by introducing some notations. 
	We say that a commutator $[x_{i_1}, \dots, x_{i_r}] \in K\langle X \rangle$ is pure if all the entries in the commutator are elements of the set $X$. A polynomial in $K\langle X \rangle$ is called proper if it is a linear combination of products of pure commutators. We denote by $B_n$ the subalgebra of $K\langle X_n \rangle$ of proper polynomials. Then $B_n$ is a $\GL(n)$-submodule of $K\langle X_n \rangle$. Let $I$ be a T-ideal in $K\langle X \rangle$ and let as before $\mathfrak{V}_I$ denote the variety of associative algebras satisfying the identities from $I$. We consider the natural quotient map
	\[
	K\left \langle X_n \right\rangle \rightarrow F_n(\mathfrak{V}_I) = K \left\langle X_n \right\rangle / (K \left\langle X_n \right\rangle \cap I).
	\]
	By $B_n(\mathfrak{V}_I)$ we denote the image of $B_n$ under the above map. If $I = T(A)$ for some PI-algebra $A$, we use also the notation $B_n(A) = B_n(\mathfrak{V}_I)$. By a theorem of Drensky (\cite{D}), we have the following decomposition of $F_n(\mathfrak{V}_I)$ as a $\GL(n)$-module:
	\begin{align} \label{eq_Drensky}
		F_n(\mathfrak{V}_I) \cong_{\GL(n)} B_n(\mathfrak{V}_I) \otimes S(V),
	\end{align}
	where again $V$ is the vector space spanned by the elements $x_1, \dots, x_n$ and $S(V)$ denotes the symmetric algebra of $V$. From Drensky's theorem it follows that $\omega_0(A)\geq 1$ for any associative unital PI-algebra $A$.

	We recall that a unital associative algebra $A$ is called Lie nilpotent of index at most $p$ if $A$ satisfies the identity $[x_1, \dots, x_{p+1}] = 0$, or in other words, if $T(A) \supseteq I_{p+1}$.
	
	Let us denote by $\mathfrak{N}_p$ the variety of all Lie nilpotent associative algebras of index at most $p$. In the above notations, $\mathfrak{N}_p = \mathfrak{V}_{I_{p+1}} = \var(I_{p+1})$, is the variety of all associative algebras satisfying the identities from $I_{p+1}$.

	The following statement now holds.

	\begin{theorem}\label{prop_Bn_finite}
		For any T-ideal $I$, the algebra $B_n(\mathfrak{V}_I)$ is a finite-dimensional subalgebra of $F_n(\mathfrak{V}_I)$ if and only if $F_n(\mathfrak{V}_I)$ satisfies the polynomial identity $[x_1, \dots, x_p]$ for some $p$, i.e., if and only if $I_p \subseteq I$ for some $p$.
	\end{theorem}
	
	\begin{proof}
		It follows from Theorem 3.3 from \cite{H} that $B_n(\mathfrak{N}_{p})$ is a finite-dimensional algebra for every $p$. Hence, for every T-ideal $I$ with $I_p \subseteq I$ we have that $B_n(\mathfrak{V}_I)$ is also a finite-dimensional algebra. Thus, the ``if'' direction of the statement is clear. To show the ``only if'' direction, we use an argument that was also used in the proof of Theorem 3.1 from \cite{DD}. For any $T$-ideal $I$ of $K\left \langle X \right\rangle$ either $F_n(\mathfrak{V}_I)$ satisfies the polynomial identity $[x_1, \dots, x_p]$ for some $p$ or $I$ is contained in $I_2I_2$. 
		Let us consider the case $I \subseteq I_2I_2$. Then we have a natural surjective homomorphism from $K\left \langle X_n \right\rangle/ (I \cap K\left \langle X_n \right\rangle)$ to $K\left \langle X_n \right\rangle/ (I_2I_2 \cap K\left \langle X_n \right\rangle)$. The image of $B_n(\mathfrak{V}_I)$ under this homomorphism is the subalgebra of proper polynomials in $K\left \langle X_n \right\rangle/ (I_2I_2 \cap K\left \langle X_n \right\rangle)$, which is known to be infinite-dimensional (see, e.g. \cite{MRZ}). Hence $B_n(\mathfrak{V}_I)$ is also infinite-dimensional.
	\end{proof}

	The above statement leads to the following corollary.
	\begin{corollary} \label{coro_w0}
		Let $A$ be a unital associative PI-algebra. Then, $\omega_0(A) = 1$ if and only if $I_{p} \subseteq T(A)$ for some positive integer $p$, i.e., if and only if $A$ is Lie nilpotent.
	\end{corollary}
	
	\begin{proof}
		Let $I = T(A)$.
		Suppose at first that there exists a positive integer $p$ such that $I_{p} \subseteq I$. Then, Theorem \ref{thm_Berele} implies that $\omega_0(A) \leq \omega_0(I_p) = 1$. Hence, this direction of the statement follows.
		
		Next, assume that there is no integer $p$ such that $I_p \subseteq I$. Then, Theorem \ref{prop_Bn_finite} implies that $B_n(\mathfrak{V}_I)$ is infinite-dimensional. 
		Let us write the $\GL(n)$-module decomposition of $B_n(\mathfrak{V}_I)$ in the following way:
		\[
		B_n(\mathfrak{V}_I) \cong_{\GL(n)} \bigoplus_{\mu} k_{\mu}(I)V_{\mu},
		\]
		where $V_{\mu}$ denotes the irreducible $\GL(n)$-module corresponding to the partition $\mu$ and by $k_{\mu}(I)$ we denote the corresponding multiplicity.
		
		Suppose on the contrary that $\omega_0(A) = 1$. This implies that there exists an integer $l$ such that for all partitions $\lambda$ which appear with nonzero multiplicity in the decomposition (\ref{eq_decomp_GL}), we have that $\lambda_2 \leq l$.
	However, the infinite dimensionality of $B_n(\mathfrak{V}_I)$ implies that there exists a partition $\mu$ with $\mu_1 \geq l+1$ such that $V_{\mu}$ appears in the decomposition of $B_n(\mathfrak{V}_I)$. Hence, by Equation (\ref{eq_Drensky}), $V_{\mu} \otimes S(V)$ appears in the decomposition of $K \left\langle X_n \right\rangle/I$. This implies that there exists $\lambda$ with $\lambda_2 \geq l+1$ such that $m_{\lambda}(I) \neq 0$. This contradiction completes the proof.
	\end{proof}

Recall that the $m$-th codimension of a PI-algebra $A$ is defined as
\[
c_m (A) = \dim P_m(A).
\]	

In \cite{GaZ}, Giambruno and Zaicev prove that if $A$ is any PI-algebra, then the limit 
\[
\exp (A) = \lim_{m \rightarrow \infty} \sqrt[m]{c_m(A)}
\]

exists and is a non-negative integer (see \cite{GaZ}, Theorem 6.5.2). This limit is called the PI exponent of $A$. In \cite{Be}, Berele shows that the following inequalities hold
\[
\omega_0(A) \leq \exp(A) \leq \omega_0(A) + \omega_1(A).
\] 

Hence, Corollary \ref{coro_w0} implies that if $A$ is unital and not Lie nilpotent and is such that $\exp(A) = 2$, then necessarily $\omega_0(A) = 2$.

\section{On the cocharacter sequence of $A$ when $T(A) = I_{p_1+1} \cdots I_{p_k+1}$} \label{sec_productOfIdeals}
	
	In this section we discuss the multiplicities $m_{\lambda}(A)$ (or $m_{\lambda}(I))$ from (\ref{eq_cochar}) and (\ref{eq_decomp_GL})  when $T(A)$ is a product of the form $I = I_{p_1 +1} \cdots I_{p_k +1}$. These multiplicities are explicitly known in several particular cases. It is well-known that when $A = U_k(K)$ is the algebra of $k \times k$ upper triangular matrices with coefficients from the field $K$, then $T(A) = \underbrace{I_2 \cdots I_2}_k$. When $k = 2$, the cocharacters of $U_2(K)$ are explicitly known (see e.g., \cite{MRZ} or \cite{BBDGK}) and the case of general $k$ was studied by Boumova and Drensky in \cite{BD} with explicit results for ``large'' partitions $\lambda$. In addition, if $A = U_k(E)$, the algebra of $k \times k$ upper triangular matrices with coefficients from the infinite-dimensional Grassmann algebra $E$, then $T(A) = \underbrace{I_3 \cdots I_3}_k$. The cocharacters of $U_2(E)$ were computed by Centrone in \cite{C}. In the same paper, Centrone computes also the cocharacters of the following matrix algebra
	
	\begin{align*}
		G = \begin{pmatrix}
			E & E \\
			0 & E^0
	\end{pmatrix},
	\end{align*}

where $E^0$ is the even part of the Grassmann algebra $E$ with the usual $\ZZ_2$-grading. For the algebra $G$ it holds that $T(G)  = I_3I_2$. The general case of cocharacters of $U_k(E)$ was considered in \cite{CDM}.
	
	In this section, we compute $\omega_0(A)$, $\omega_1(A)$ and $s_2(A)$ from Theorem \ref{thm_Berele} in the general case when $T(A) = I_{p_1 +1} \cdots I_{p_k+1}$.
	
	We recall that if $R = \bigoplus_{\mu = (\mu_1, \dots, \mu_n) \in \ZZ^n_{\geq 0}} R^{\mu}$ is an algebra with a $\ZZ^n_{\geq 0}$-grading we define the Hilbert series of $R$ with respect to the $\ZZ^n_{\geq 0}$-grading to be the following formal power series
	\[
	H(R, t_1, \dots, t_n) = \sum_{\mu \in \ZZ^n_{\geq 0}} \operatorname{dim} R^{\mu} t_1^{\mu_1} \cdots t_n^{\mu_n}.
	\]
	
	This notion generalizes the standard notion of a Hilbert (or Poincar\'{e}) series for an algebra with a $\ZZ$-grading.
	
	For shortness, we denote $T = \{t_1, \dots, t_n\}$ and we write for the Hilbert series $H(R, T) = H(R, t_1, \dots, t_n)$.
	
	The relatively free algebra $F_n(\mathfrak{V}_I) = K\left \langle X_n \right\rangle/ (I \cap K\left \langle X_n \right\rangle)$ has a $\ZZ^n_{\geq 0}$-grading which comes from the weight space decomposition as a $\GL(n)$-module and hence its Hilbert series coincides with its character as a $\GL(n)$-module. In other words, if 
	\[
	F_n(\mathfrak{V}_I) \cong_{\GL(n)} \bigoplus_{\lambda = (\lambda_1, \dots, \lambda_n)} m_{\lambda}(I) V_{\lambda}
	\]
	as in Equation (\ref{eq_decomp_GL}), then for the Hilbert series of $F_n(\mathfrak{V}_I)$ we have the expression
	\[
	H(F_n(\mathfrak{V}_I), T) = \sum_{\lambda = (\lambda_1, \dots, \lambda_n)} m_{\lambda}(I) S_{\lambda}(T),
	\]
	where $S_{\lambda}(T) = S_{\lambda}(t_1, \dots, t_n)$ denotes the Schur polynomial corresponding to the partition $\lambda$. Thus, the Hilbert series $H(F_n(\mathfrak{V}_I), T)$ gives useful information about the multiplicities $m_{\lambda}(I)$. 
	
	Following \cite{C}, we define also the proper Hilbert series of $F_n(\mathfrak{V}_I)$ to be the Hilbert series of the subalgebra of proper polynomials in $F_n(\mathfrak{V}_I)$, i.e.,
	\[
	H^B(F_n(\mathfrak{V}_I), T) = H(B_n(\mathfrak{V}_I), T).
	\]
	
	By Equation (\ref{eq_Drensky}), we have the following relation between the Hilbert series and the proper Hilbert series of a relatively free algebra:
	
	\begin{align} \label{eq_properHilbertSeries}
		H(F_n(\mathfrak{V}_I), T) = \prod_{i = 1}^n\frac{1}{1-t_i}H^B(F_n(\mathfrak{V}_I), T).
	\end{align}

	Next, when a $T$-ideal $I$ is a product of two $T$-ideals $I = J_1J_2$, then the following relation between the Hilbert series was given by Formanek (\cite{F}):
	
	\begin{equation} \label{eq_ProductHilbertSeries}
		\begin{aligned}
			H(F_n(\mathfrak{V}_I), T) = &H(F_n(\mathfrak{V}_{J_1}), T) + H(F_n(\mathfrak{V}_{J_2}), T) + \\
			&(S_{(1)}(T) - 1)H(F_n(\mathfrak{V}_{J_1}), T)H(F_n(\mathfrak{V}_{J_2}), T),
		\end{aligned}
	\end{equation}
	
	where $S_{(1)}(t_1, \dots, t_n) = t_1+ \dots + t_n$ denotes the Schur polynomial corresponding to the partition $\lambda = (1, 0, \dots, 0)$.
	
	When a $T$-ideal $I$ is a product of several $T$-ideals $I = J_1 \cdots J_k$ for some $k \geq 2$, then we can iterate Equation (\ref{eq_ProductHilbertSeries}) to obtain the following:
	
	\begin{equation} \label{eq_ProductManyHilbertSeries}
		\begin{aligned}
			H(F_n(\mathfrak{V}_I), T) = \sum_{r = 1}^k(S_{(1)}-1)^{r-1} \sum_{1 \leq i_1 < \cdots <i_r \leq k} H(F_n(\mathfrak{V}_{J_{i_1}}), T) \cdots H(F_n(\mathfrak{V}_{J_{i_r}}), T).
		\end{aligned}
	\end{equation}		
	
	Using Equations (\ref{eq_properHilbertSeries}) and (\ref{eq_ProductHilbertSeries}), when $I = J_1J_2$ one can easily derive a relation between the proper Hilbert series of $F_n(\mathfrak{V}_I)$ on the one hand and the proper Hilbert series of $F_n(\mathfrak{V}_{J_1})$ and $F_n(\mathfrak{V}_{J_2})$ on the other hand. The relation is as follows (see also \cite{C}):
	
	\begin{equation}\label{eq_productProperHilbertSeries}
		\begin{aligned}
			H^B(F_n(\mathfrak{V}_I), T) = &H^B(F_n(\mathfrak{V}_{J_1}), T) + H^B(F_n(\mathfrak{V}_{J_2}), T) + \\
			&\frac{S_{(1)}(T) - 1}{\prod_{i=1}^n(1-t_i)}H^B(F_n(\mathfrak{V}_{J_1}),T)H^B(F_n(\mathfrak{V}_{J_2}), T).
		\end{aligned}
	\end{equation}

	Formula (\ref{eq_ProductHilbertSeries}) together with Littlewood-Richardson rule for multiplication of Schur polynomials lead to the following observation.
	
	\begin{proposition} \label{prop_product}
		Let $A$, $B_1$, and $B_2$ be PI-algebras such that $T(A) = T(B_1)T(B_2)$. 
		Let $\omega_0$, $\omega_1$, $s_1$ and $s_2$ be as before. Then, the following relations hold:
		\begin{itemize}
			\item[(1)] $\omega_0(A) = \omega_0(B_1) + \omega_0(B_2)$;
			\item[(2)] $\omega_1(A) = \omega_1(B_1) + \omega_1(B_2)$;
			\item[(3)] $s_1(A) = s_1(B_1) + s_1(B_2) + 1$;
			\item[(4)] $s_2(A) = s_2(B_1) + s_2(B_2) +1$;
		\end{itemize}
	\end{proposition}

\begin{proof}
	We will prove first part (1) and part (4). 
	Let us denote $I = T(A)$ and $J_i = T(B_i)$ for $i = 1,2$. To show that $\omega_0(A) \geq \omega_0(B_1) + \omega_0(B_2)$ we need to show that for each positive integer $\ell$, there exists a partition $\lambda$ with $\lambda_{\omega_0(B_1) + \omega_0(B_2)} > \ell$ such that $m_{\lambda}(A) \neq 0$. In other words, the Schur function $S_{\lambda}(T)$ appears with nonzero multiplicity in the Hilbert series $H(F_n(\mathfrak{V}_I), T)$. We take partitions $\mu$ and $\nu$ such that $\mu_{\omega_0(B_1)} > \ell$ and $\nu_{\omega_0(B_2)} > \ell$ and such that $S_{\mu}(T)$ appears with nonzero multiplicity in the Hilbert series $H(F_n(\mathfrak{V}_{J_1}), T)$ and $S_{\nu}(T)$ appears with nonzero multiplicity in the Hilbert series $H(F_n(\mathfrak{V}_{J_2}), T)$. We can do that by the definition of $\omega_0$. Then, Littlewood-Richardson rule implies that in the product $S_{\mu}(T)S_{\nu}(T)$ there will be a polynomial $S_{\lambda}(T)$ with the desired properties of $\lambda$. 
	
	To show that $\omega_0(A) \leq \omega_0(B_1) + \omega_0(B_2)$  and prove also part (4) we need to show that for every partition $\lambda$ for which $S_{\lambda}(T)$ appears with nonzero multiplicity in $H(F_n(\mathfrak{V}_I), T)$, it holds that $\lambda_{\omega_0(B_1) + \omega_0(B_2) + 1} \leq s_2(B_1) + s_2(B_2) + 1$. 
	
	Let us take a partition $\lambda$, such that $S_{\lambda}(T)$ appears with nonzero multiplicity in $H(F_n(\mathfrak{V}_I), T)$. Then, Formula (\ref{eq_ProductHilbertSeries}) implies that $S_{\lambda}(T)$ appears with nonzero multiplicity at least in one of the expressions $H(F_n(\mathfrak{V}_{J_1}), T)$, $H(F_n(\mathfrak{V}_{J_2}), T)$ or $(S_{(1)}(T) - 1)H(F_n(\mathfrak{V}_{J_1}), T)H(F_n(\mathfrak{V}_{J_2}), T)$. If $S_{\lambda}(T)$ appears with nonzero multiplicity in $H(F_n(\mathfrak{V}_{J_i}), T)$ for $i=1,2$, then $\lambda_{\omega_0(B_i)+1} \leq s_2(B_i)$ which implies the desired property. 
	
	Let us now consider the case when $S_{\lambda}(T)$ appears with nonzero multiplicity in $S_{(1)}(T)H(F_n(\mathfrak{V}_{J_1}), T)H(F_n(\mathfrak{V}_{J_2}), T)$. Then, $S_{\lambda}(T)$ appears in a product of Schur polynomials of the form $S_{(1)}(T)S_{\mu}(T)S_{\nu}(T)$, such that $S_{\mu}(T)$ appears with nonzero multiplicity in $H(F_n(\mathfrak{V}_{J_1}), T)$ and $S_{\nu}(T)$ appears with nonzero multiplicity in $H(F_n(\mathfrak{V}_{J_2}), T)$. By using again Littlewood-Richardson rule we prove the desired property.	
	
	To prove parts (2) and (3), we use the following symmetry of the Littlewood-Richardson rule (which can be deduced e.g. from \cite{Mc}, Chapter I.3): Let $\lambda$, $\mu$, and $\nu$ be partitions and let $\lambda'$, $\mu'$, and $\nu'$ denote the conjugate partitions. If we have that
	\[
	S_{\mu}(T) S_{\nu}(T) = \sum_{\lambda} c^{\lambda}_{\mu, \nu} S_{\lambda}(T) 
	\]
	
	and
	\[
	S_{\mu'}(T) S_{\nu'}(T) = \sum_{\lambda'} c^{\lambda'}_{\mu', \nu'} S_{\lambda'}(T) 
	\]
	
	then $c^{\lambda}_{\mu, \nu} = c^{\lambda'}_{\mu', \nu'}$.
	
	To show that $\omega_1(A) \geq \omega_1(B_1) + \omega_1(B_2)$ we need to show that for each positive integer $\ell$, there exists a partition $\lambda$ with $\lambda'_{\omega_1(B_1) + \omega_1(B_2)} > \ell$ such that $m_{\lambda}(A) \neq 0$. We take partitions $\mu$ and $\nu$ such that $\mu'_{\omega_1(B_1)} > \ell$ and $\nu'_{\omega_1(B_2)} > \ell$ and such that $S_{\mu}(T)$ appears with nonzero multiplicity in the Hilbert series $H(F_n(\mathfrak{V}_{J_1}), T)$ and $S_{\nu}(T)$ appears with nonzero multiplicity in the Hilbert series $H(F_n(\mathfrak{V}_{J_2}), T)$. We can do that by the definition of $\omega_1$. Then, in the product $S_{\mu'}(T)S_{\nu'}(T)$ there will be a polynomial $S_{\lambda'}(T)$ with $\lambda'_{\omega_1(B_1) + \omega_1(B_2)} > \ell$. By the above mentioned symmetry of the Littlewood-Richardson rule, if $S_{\lambda'}(T)$ appears in the product $S_{\mu'}(T)S_{\nu'}(T)$, then $S_{\lambda}(T)$ will appear in the product $S_{\mu}(T)S_{\nu}(T)$ and hence $m_{\lambda}(A) \neq 0$. This shows that $\omega_1(A) \geq \omega_1(B_1) + \omega_1(B_2)$.
	The rest of part (2) and part (3) is proven analogously.
\end{proof}	
	
	Iterating the above proposition we obtain:
	
	\begin{proposition} \label{prop_multProduct}
		Let $A$, $B_1, \dots, B_k$ for some $k \geq 2$ be PI-algebras with $T(A) = T(B_1)\cdots T(B_k)$.
		Then
		\begin{itemize}
			\item $\omega_0(A) = \omega_0(B_1) + \cdots + \omega_0(B_k)$;
			\item $\omega_1(A) = \omega_1(B_1) + \cdots + \omega_1(B_k)$;
			\item $s_1(A) = s_1(B_1) + \cdots + s_1(B_k) + k-1$;
			\item $s_2(A) = s_2(B_1) + \cdots + s_2(B_k) + k-1$;
		\end{itemize}
	\end{proposition}

Proposition \ref{prop_multProduct} together with Theorem \ref{thm_Berele} and the considerations after it imply the following corollary.
	
	\begin{corollary} \label{coro_multProduct}
		Let $A$ be a PI-algebra with $T(A) = I$ such that $I$ is a product of $k$ ideals $I = I_{p_1+1} \cdots I_{p_k+1}$ for some positive integers $p_1, \dots, p_k$. Let also $k < n$. Then the following hold:
		\begin{itemize}
			\item $\omega_0(A) = k$;
			\item $\omega_1(A) = k$ (This is in the case when $p_i \geq 2$ for all $i=1, \dots, k$);
			\item $s_2(A) = p_1 + \cdots + p_k -1$;
		\end{itemize}
	\end{corollary}
	
	In particular, Corollary \ref{coro_multProduct} implies the following. If $I = I_{p_1+1} \cdots I_{p_k+1}$ and $\lambda$ is a partition such that $m_{\lambda}(I) \neq 0$ then $\lambda_{k+1} \leq p_1 + \cdots + p_k -1$. 
	Using Equation (\ref{eq_ProductManyHilbertSeries}), we obtain a further restriction on the partition $\lambda$.
	
	\begin{proposition} \label{prop_bound_lambda}
		Let $I = I_{p_1+1} \cdots I_{p_k+1}$ and let $\lambda$ be a partition such that $m_{\lambda}(I) \neq 0$. Then, $\lambda_{k+1} \leq p_1 + \cdots + p_k -1$ and $\lambda_{2k} \leq p_1 + \cdots + p_k -k$.
\end{proposition}

\begin{proof}
	We need to show that $\lambda_{2k} \leq p_1 + \cdots + p_k -k$. We choose an integer $r$ between $1$ and $k$. Let $\lambda$ be a partition such that $S_{\lambda}(T)$ appears with non-zero multiplicity in the product
	\[
	(S_{(1)})^{r-1} \sum_{1 \leq i_1 < \cdots <i_r \leq k} H(F_n(\mathfrak{V}_{J_{i_1}}), T) \cdots H(F_n(\mathfrak{V}_{J_{i_r}}), T).
	\]
	
	Therefore, there exist integers $1 \leq i_1 < \cdots <i_r \leq k$ and  partitions $\mu_{i_1}, \dots, \mu_{i_r}$ such that $\mu_{i_j} \in H(1, p_{i_j} -1)$ and such that $S_{\lambda}(T)$ appears with non-zero multiplicity in the product 
	\[
	(S_{(1)})^{r-1} S_{\mu_{i_1}}(T) \cdots S_{\mu_{i_r}}(T).
	\] 
	Hence, $\lambda$ is obtained by adding $r-1$ boxes to a partition $\tilde{\lambda}$ which is contained in the hook $H(r, p_{i_1}+ \cdots + p_{i_r} - r)$. Thus, $\lambda_{r+1} \leq p_{i_1} + \cdots + p_{i_r} - 1$ and $\lambda_{2r} \leq  p_{i_1}+ \cdots + p_{i_r} - r$.
	
	We can do these considerations for each $r$ between $1$ and $k$ and when $r = k$ we obtain the statement of the proposition.	
\end{proof}

	The fact that the subalgebra of proper polynomials $B_n(\mathfrak{N}_p)$ in $F_n(\mathfrak{N}_p)$ is finite-dimensional for every $p \geq 1$ implies that the proper Hilbert series $H^B(F_n(\mathfrak{N}_p), T)$ is a polynomial. Therefore, using Formula (\ref{eq_productProperHilbertSeries}) and the relation between the proper Hilbert series and the Hilbert series given by Formula (\ref{eq_properHilbertSeries}) we obtain a general form for the Hilbert series and the proper Hilbert series of any relatively free algebra $F_n(\mathfrak{V}_I)$ where $I$ is a $T$-ideal of the form $I=I_{p_1+1}\cdots I_{p_k+1}$.

	\begin{proposition} \label{prop_MultProductProper}
		Let $I$ be a $T$-ideal such that $I$ is a product of $k$ ideals $I = I_{p_1+1} \cdots I_{p_k+1}$ for some positive integers $p_1, \dots, p_k$. Let also $k < n$. Then the following hold:
		\begin{itemize}
			\item [(i)]
			\[
			H^B(F_n(\mathfrak{V}_I), T) = P_1(T) + \frac{P_2(T)}{\prod_{i= 1}^n(1-t_i)} + \frac{P_3(T)}{\prod_{i= 1}^n(1-t_i)^2} + \dots + \frac{P_k(T)}{\prod_{i= 1}^n(1-t_i)^{k-1}};
			\]
			
			\item [(ii)]
			\[
			H(F_n(\mathfrak{V}_I), T) = \frac{P_1(T)}{\prod_{i= 1}^n(1-t_i)} + \frac{P_2(T)}{\prod_{i= 1}^n(1-t_i)^2} + \frac{P_3(T)}{\prod_{i= 1}^n(1-t_i)^3} + \dots + \frac{P_k(T)}{\prod_{i= 1}^n(1-t_i)^{k}},
			\]
		\end{itemize}
		where $P_1(T), \dots, P_k(T)$ are the following polynomials
		
		\begin{align*}
			&P_1(T) = \sum_{i= 1}^k H^B(F_n(\mathfrak{N}_{p_i}), T);\\
			&P_2(T) = (S_{(1)}(T) -1) \sum_{1 \leq i < j \leq k}H^B(F_n(\mathfrak{N}_{p_i}), T)H^B(F_n(\mathfrak{N}_{p_j}), T);\\
			&\dots \dots \\
			&P_r(T) =(S_{(1)}(T) -1)^{r-1} \sum_{1 \leq i_1 < \dots < i_r \leq k}H^B(F_n(\mathfrak{N}_{p_{i_1}}), T)\cdots H^B(F_n(\mathfrak{N}_{p_{i_r}}), T);\\
			&\dots \dots \\
			&P_k(T) = (S_{(1)}(T) -1)^{k-1}H^B(F_n(\mathfrak{N}_{p_{1}}), T)\cdots H^B(F_n(\mathfrak{N}_{p_{k}}), T).
		\end{align*}
		
	\end{proposition}

	As we already mentioned above, $T$-ideals $I$ that are products of the form $I = \underbrace{I_2 \cdots I_2}_{k \text { times }}$ and the corresponding cocharacters were studied by Boumova and Drensky in \cite{BD} with explicit results for ``large'' partitions $\lambda$ (see \cite{BD} for the precise definition of ``large''). We notice that when $I = \underbrace{I_2 \cdots I_2}_{k \text { times }}$, i.e., when $p_1 = \dots = p_k = 1$, we have that $H^B(F_n(\mathfrak{N}_{p_{i_r}}), T) = 1$ for all $r= 1, \dots, k$. Therefore, for the polynomials $P_r(T)$ for $r= 1, \dots, k$, using Proposition \ref{prop_MultProductProper}, we obtain the following:
	\begin{align*}
		P_r(T) =&(S_{(1)}(T) -1)^{r-1} \sum_{1 \leq i_1 < \dots < i_r \leq k}H^B(F_n(\mathfrak{N}_{p_{i_1}}), T)\cdots H^B(F_n(\mathfrak{N}_{p_{i_r}}), T) = \\
		&(S_{(1)}(T) -1)^{r-1}\binom{k}{r}.
	\end{align*}
	
	Therefore, Proposition \ref{prop_MultProductProper} leads to the following expression for the Hilbert series of $F_n(\mathfrak{V}_{I})$, which coincides with the results from \cite{BD}.
	\[
	H(F_n(\mathfrak{V}_{I}), T) = \sum_{r= 1}^k \binom{k}{r} \frac{(S_{(1)}(T) -1)^{r-1}}{\prod_{i=1}^n(1-t_i)^r}.
	\]

	\section{Applications to the cocharacters of Lie nilpotent associative algebras} \label{sec_app}
	
	We recall the following well-known result on product of commutators:
	
	\begin{theorem} [\cite{BJ}, \cite{DeK2}]\label{thm_product2}
		For positive integers $m_1$ and $m_2$ such that at least one of them is odd it holds that
		\[
		I_{m_1}I_{m_2} \subset I_{m_1 + m_2 - 1}.
		\]
	\end{theorem}
	
	Iterating the above result we obtain the following statement (see also \cite{Da} and \cite{DeK2}). 
	\begin{proposition} For any $k \geq 2$, if at most one of the positive integers $m_1, \dots, m_k$ is even then
		\[
		I_{m_1} \cdots I_{m_k} \subset I_{m_1 + \dots + m_k -k +1}.
		\]
	\end{proposition}
	
	Therefore, if a $T$-ideal $I$ is expressed as $I = I_{p_1+1}\cdots I_{p_k+1}$, where at most one of $p_1, \dots, p_k$ is a positive odd integer, we have a natural surjective homomorphism of relatively free algebras
	\[
	f: K\left\langle X_n \right\rangle/I_{p_1+1}\cdots I_{p_k+1} \rightarrow K\left\langle X_n \right\rangle/I_{p_1+ \dots + p_k+1}.
	\]
	and this homomorphism restricts naturally to the respective subalgebras of proper polynomials $B_n(\mathfrak{V}_{I_{p_1+1}\cdots I_{p_k+1}})$ and $B_n(\mathfrak{N}_{p_1 + \dots +p_k})$. The map $f$ is $\GL(n)$-equivariant. By Corollary \ref{coro_multProduct} and Proposition \ref{prop_bound_lambda}, if $m_{\lambda}(I) \neq 0$ for some partition $\lambda = (\lambda_1 \geq \dots \geq \lambda_n)$, then $\lambda_1, \dots, \lambda_k$ can be arbitrarily large, $\lambda_{k+1} \leq p_1 + \dots + p_k -1$, and $\lambda_{2k} \leq p_1 + \dots + p_k -k$. 
	This observation leads to the following proposition, which is an improvement of Theorem 3.3 from \cite{H}.

	\begin{proposition} \label{prop_bound_Ip}
		Consider the relatively free algebra $F_n(\mathfrak{N}_p)$ and let $\lambda$ be a partition such that $m_{\lambda}(I_{p+1}) \neq 0$ in the decomposition
		\[
		F_n(\mathfrak{N}_p) \cong_{\GL(n)} \bigoplus_{\lambda}m_{\lambda}(I_{p+1})V_{\lambda}.	
		\]
		 Then, for each $1 \leq k \leq \lfloor \frac{n}{2} \rfloor $ such that $p > 2(k-1)$ it holds that $\lambda_{2k} \leq p-k$.
	\end{proposition} 
	
	\begin{proof}
		For $k=1$ the statement follows from Theorem 3.3 from \cite{H}.
		Next, we fix one $k$ such that $2 \leq k \leq \lfloor \frac{n}{2} \rfloor $ and $p > 2(k-1)$. Then, the integer $p$ can be represented in the form 
		\[
		p = \underbrace{2 + \dots + 2}_{k-1} + p - 2(k-1).
		\]
		
		Hence, if we set $p_1 = 2$, $\dots$, $p_{k-1} = 2$ and $p_k = p-2(k-1)$ we have that $p_1 + \dots + p_k = p$. Therefore, $I_{p_1+1} \cdots I_{p_k +1 } \subset I_{p+1}$. Hence, the restriction $\lambda_{2k} \leq p-k$ follows from the observations before the statement.
	\end{proof}
	
	If we take $p=3$ and consider the relatively free algebra $F_n(\mathfrak{N}_3)$ then Proposition \ref{prop_bound_Ip} implies the following: If $m_{\lambda}(I_4) \neq 0$ then $\lambda_2 \leq 2$ and $\lambda_4 \leq 1$. These results agree with the $S_m$-module decomposition of $P_m(\mathfrak{N}_3)$, which can be deduced from the results of Volichenko in \cite{V}. Moreover, both bounds $\lambda_2 \leq 2$ and $\lambda_4 \leq 1$ are obtained, however not simultaneously. This means that if $m_{\lambda}(I_4) \neq 0$ and if $\lambda_2 = 2$ then $\lambda_4 = 0$. And in addition, if $m_{\lambda}(I_4) \neq 0$ with $\lambda_4 = 1$ then necessarily $\lambda_2 =1$. 
	
	Similarly, if we take $p=4$, then for $F_n(\mathfrak{N}_4)$ we obtain that if $m_{\lambda}(I_5) \neq 0$ then $\lambda_2 \leq 3$ and $\lambda_4 \leq 2$. These results again agree with the known $S_m$-module decomposition of $P_m(\mathfrak{N}_4)$, which was obtained by Stoyanova-Venkova in \cite{SV}. However, in this case the bound $\lambda_4 \leq 2$ is not obtained. That is, for all partitions $\lambda$ with $m_{\lambda}(I_5) \neq 0$ we see from \cite{SV} that $\lambda_4 \leq 1$.
	
	We also notice that for $p=5$, i.e., when we consider $F_n(\mathfrak{N}_5)$, then Proposition \ref{prop_bound_Ip} gives already three restrictions on $\lambda$. Namely, if $m_{\lambda}(I_6) \neq 0$ then $\lambda_2 \leq 4$, $\lambda_4 \leq 3$, and $\lambda_6 \leq 2$.
	
	For the conjugate partition $\lambda'$, using similar arguments we obtain the following bounds.
	
	\begin{proposition} \label{prop_bound_conjugate}
		Consider again the relatively free algebra $F_n(\mathfrak{N}_p)$ and let $\lambda$ be a partition such that $m_{\lambda}(I_{p+1}) \neq 0$ in the decomposition
		\[
		F_n(\mathfrak{N}_p) \cong_{\GL(n)} \bigoplus_{\lambda}m_{\lambda}(I_{p+1})V_{\lambda}.	
		\]
		Then, the following hold:
		\begin{itemize}
		\item if $p = 2k$, then $\lambda'_{k+1} \leq 2k-1$;
		\item if $p = 2k+1$, then $\lambda'_{k+1} \leq 2k+1$. 
		\end{itemize}
	\end{proposition}

\begin{proof}
	Consider first the case $p = 2k$. Then $p = p_1 + \dots + p_k$, where $p_1 = \dots = p_k = 2$. Hence we have a surjective homomorphism
	\[
	f: K \left\langle X_n \right\rangle / \underbrace{I_3 \cdots I_3}_{k} \rightarrow K \left\langle X_n \right\rangle /I_{p+1}.
	\]
	
	The cocharacter sequence of the infinite-dimensional Grassmann algebra $E$ ( which has $T(E) = I_3$) is well-known and from there we have that $s_1(E) = 1$. Therefore, Proposition \ref{prop_multProduct} imlies that if $A$ is such that $T(A) = \underbrace{I_3 \cdots I_3}_{k}$, then $s_1(A) = 2k-1$. Therefore, since $\omega_1(A) = k$, the above surjective map $f$ shows that if $\lambda$ is a partition with $m_{\lambda}(I_{p+1}) \neq 0$, then $\lambda'_{k+1} \leq 2k-1$. 
	
	If $p = 2k+1$, then $p = p_1 + \dots + p_{k-1} + p_k$, where $p_1 = \dots = p_{k-1} = 2$ and $p_k = 3$. Hence we have a surjective homomorphism
	\[
	f_1: K \left\langle X_n \right\rangle / \underbrace{I_3 \cdots I_3}_{k-1}I_4 \rightarrow K \left\langle X_n \right\rangle /I_{p+1}.
	\]
	
	By the results of Volichenko \cite{V}, we have that if $B$ is an algebra with $T(B) = I_4$, then $s_1(B) = 3$. Hence, if $A$ is such that $T(A) = \underbrace{I_3 \cdots I_3}_{k-1}I_4$, then $s_1(A) = 2k+1$. Hence, the surjective map $f_1$ shows that if $\lambda$ is a partition with $m_{\lambda}(I_{p+1}) \neq 0$, then $\lambda'_{k+1} \leq 2k+1$. 
\end{proof}

	\section{Applications to Noncommutative invariant theory} \label{sec_invariantTheory}
	
	The importance of the numbers $\omega_0(A)$, $\omega_1(A)$, $s_1(A)$, and $s_2(A)$ can be seen also when considering algebras of invariants. In this section we consider the algebra of $\SL(n)$-invariants $(F_n(\mathfrak{V}_I))^{\SL(n)}$, for an arbitrary $T$-ideal $I$. The following theorem can be deduced from \cite{BBDGK}.
	
	\begin{theorem}\cite{BBDGK} \label{thm_BBDGK_SL}
		Let $R$ be a finitely generated graded algebra over $K$ which is also a $\GL(n)$-module with the following decomposition 
		\[ 
		R = \bigoplus_{i\geq 0} R^i = \bigoplus_{i \geq 0} \bigoplus_{\lambda} m_i(\lambda) V_{\lambda}.
		\]
		Let $G = \SL(n)$. Then the Hilbert series of the algebra of invariants $R^G$ is given by
		\[ 
		H(R^G, t) = \sum_{i \geq 0} (\sum_{\lambda}m_i(\lambda))t^i,
		\]
		where the second sum is over all $\lambda$ with $\lambda_1 = \lambda_2 = \dots = \lambda_n$;
	\end{theorem}	
	
	When we apply the above theorem to the case $R = F_n(\mathfrak{V}_I)$ for some $T$-ideal $I$, we obtain the following corollary.
	
	\begin{corollary} \label{coro_FinDimSL} Let $A$ be a PI-algebra with $T(A) = I$ and let the integers $\omega_0(A)$, $\omega_1(A)$, $s_1(A)$ and $s_2(A)$ be as before. If $n > \omega_0(A)$, then $(F_n(\mathfrak{V}_I))^{\SL(n)}$ is finite-dimensional. In particular, the Hilbert series $H((F_n(\mathfrak{V}_I))^{\SL(n)}, t)$ is a polynomial and
		\begin{align*}
			\deg H((F_n(\mathfrak{V}_I))^{\SL(n)}, t) \leq ns_2(A).
		\end{align*}
		
		Furthermore, if $n > s_1(A)$ then
		\begin{align*}
			\deg H((F_n(\mathfrak{V}_I))^{\SL(n)}, t) \leq n\omega_1(A).
		\end{align*}
	\end{corollary}	
	
	\begin{proof}
		We denote by $(F_n(\mathfrak{V}_I))^{(i)}$ the homogeneous component of total degree $i$ in $F_n(\mathfrak{V}_I)$. Then we can write the decomposition of $F_n(\mathfrak{V}_I)$ as a $\GL(n)$-module in the following way:
		\[
		F_n(\mathfrak{V}_I) = \bigoplus_{i \geq 0} (F_n(\mathfrak{V}_I))^{(i)}\cong_{\GL(n)} \bigoplus_{i \geq 0} \bigoplus_{\lambda} m_{i, \lambda} V_{\lambda},
		\]
		where $|\lambda| = i$ and by the definitions of the numbers $\omega_0$, $\omega_1$, $s_1$ and $s_2$, $\lambda_{\omega_0(A) + 1} \leq s_2(A)$ and $\lambda_{s_1(A) +1} \leq \omega_1(A)$.
		
		Then Theorem \ref{thm_BBDGK_SL} implies that
		\[
		H((F_n(\mathfrak{V}_I))^{\SL(n)}, t) = \sum_{i \geq 0}\left(\sum_{\lambda} m_{i, \lambda}\right)t^i,
		\]
		where the second sum runs over partitions $\lambda$ with $\lambda_1 = \lambda_2 = \dots = \lambda_{n}$ such that $|\lambda| = i$,  $\lambda_{\omega_0(A) + 1} \leq s_2(A)$ and $\lambda_{s_1(A) +1} \leq \omega_1(A)$. Hence, if $\omega_0(A) < n \leq s_1(A)$ the largest partition $\lambda$ which satisfies these conditions is $\lambda_1 = \lambda_2 = \dots = \lambda_{n} = s_2(A)$. Therefore, the highest power of $t$ which can occur in the Hilbert series $H((F_n(\mathfrak{V}_I))^{\SL(n)}, t)$ is $i = ns_2(A)$. Similarly, if $n > s_1(A)$, then the largest partition $\lambda$ which satisfies the necessary conditions is $\lambda_1 = \lambda_2 = \dots = \lambda_{n} = \omega_1(A)$. Thus in this case the highest power of $t$ which can occur in the Hilbert series $H((F_n(\mathfrak{V}_I))^{\SL(n)}, t)$ is $i = n\omega_1(A)$. This completes the proof.
	\end{proof}

	When we apply the above corollary to the case when $I = I_{p_1+1}\cdots I_{p_k+1}$ we obtain the following.
	
	\begin{corollary}
		Let $I = I_{p_1 + 1} \cdots I_{p_k +1}$ for some positive integers $p_1, \dots, p_k$. If $k \leq n-1$, then the algebra of invariants $(F_n(\mathfrak{V}_I))^{\SL(n)}$ is finite-dimensional and
		\[
		\deg H((F_n(\mathfrak{V}_I))^{\SL(n)}, t) \leq n(p_1 + \dots + p_k -k).
		\]
	\end{corollary}
	
	\section*{Acknowledgements}
	I am grateful to Vesselin Drensky for the useful discussions that we had during my work on this paper and for reading the first draft of the paper.

\end{document}